\documentclass[reqno]{amsart}

\usepackage{amsmath,amssymb}
\usepackage{amsfonts}
\usepackage{amsthm}
\usepackage{enumerate}
\usepackage{fancyhdr}

\makeatletter
\newcommand*\owedge{\mathpalette\@owedge\relax}
\newcommand*\@owedge[1]{%
  \mathbin{%
    \ooalign{%
      $#1\m@th\bigcirc$\cr
      \hidewidth$#1\m@th\wedge$\hidewidth\cr
    }%
  }%
}
\makeatother

\DeclareMathOperator{\vol}{vol}

\newcommand{\Rm}{\mathrm{Rm}}
\newtheorem{Thm}{Theorem}[section]

\newtheorem{Cor}{Corollary}[section]
\newtheorem{Lem}{Lemma}[section]

\newtheorem{Def}{Definition}[section]
\numberwithin{equation}{section}
\newtheorem{theoL}{Theorem}
\newtheorem*{theoC'}{Theorem C'}
\newtheorem*{ThmG}{Gromov--Ruh Theorem}

\begin{document}

\title{Ricci Flow and Gromov Almost Flat Manifolds}
\author{Eric Chen, Guofang Wei, and Rugang Ye}
\address{Department of Mathematics, University of California, Santa Barbara CA 93106-3080, USA}
\curraddr{Department of Mathematics, University of California, Berkeley CA 94720-3840, USA}
\email{ecc@math.berkeley.edu}
\address{Department of Mathematics, University of California, Santa Barbara CA 93106-3080, USA}
\email{wei@ucsb.edu}
\address{Department of Mathematics, University of California, Santa Barbara CA 93106-3080, USA}
\email{yer@ucsb.edu}
\thanks{E. Chen is partially supported by NSF Grant DMS 3103392, and  partially supported by an AMS--Simons Travel Grant. }
\thanks{G. Wei is partially supported by NSF Grant DMS 2104704.}

\begin{abstract}
We employ the Ricci flow to derive a new theorem about Gromov almost flat manifolds, which generalizes and strengthens the celebrated Gromov--Ruh Theorem. In our theorem, the condition $diam^2 |K| \leq \epsilon_n$ in the Gromov--Ruh Theorem is replaced by the substantially weaker condition $\|Rm\|_{n/2}$ $ C_S^2 \leq \varepsilon_n$.
\end{abstract}

\date{\today}
\maketitle

\section{Introduction}

In this paper all manifolds, Riemannian manifolds and Riemannian metrics are assumed to be smooth. In the late 1970s M.~Gromov introduced the concept of almost flat manifolds.
\vspace{1ex}

\begin{Def}[\cite{Gromov1978a}]~
\begin{enumerate}[1)]
\item Let $\epsilon>0$. A Riemannian manifold $(M, g)$ (or the metric $g$) is called $\epsilon$-flat, provided that $diam^2 \sup_M |K| \leq \epsilon$, where $K$ denotes sectional curvature.
\item A compact manifold $M$ is called {\it almost flat}, provided that there is a sequence of Riemannian metrics 
$g_k$ on $M$ such that $diam_{g_k}^2 \sup_M  |K_{g_k}| \rightarrow 0$, i.~e.~
$g_k$ is $\epsilon_k$-flat, and $\epsilon_k \rightarrow 0$.
\end{enumerate}
\end{Def}

Note that $K$ can be replaced by $|Rm|$ in the above definition, where $Rm$ denotes the Riemann curvature tensor. Note also that 
$\sup_M |Rm| =\|Rm\|_{C^0(M)}=\|Rm\|_{L^{\infty}(M)}$. The crowning achievement of Gromov's theory of almost flat manifolds is the following celebrated theorem.
\vspace{1ex}

\begin{ThmG}[\cite{Gromov1978a,Ruh1982}]~
\begin{enumerate}[1)]
\item For each $n \geq 3$ there exists a positive constant $\epsilon_n$ with the  following property.
A compact manifold $M$ of dimension $n$ is diffeomorphic to an infranil manifold if and only it admits a Riemannian metric 
satisfying  $diam^2 \sup_M |K| \leq \epsilon_n$.
\item A compact manifold is almost flat if and only if it is diffeomorphic to an infranil manifold.
\end{enumerate}
\end{ThmG}

Here one can assume $\epsilon_n \leq 1$. We recall the definition of infranil manifolds.

\begin{Def} Let $N$ be a simply connected nilpotent Lie group $N$ acting on itself by left multiplication, and  $F$ a finite group of automorphisms  of $N$.  A smooth action of the semi-direct product $ N\rtimes F$ is then defined on $N$. An orbit space of $N$ by a discrete subgroup of $ N\rtimes F$ which acts freely on $N$ is called an infranil manifold. An infranil manifold is finitely covered by a nilmanifold.
\end{Def}

Note that flat manifolds, i. e. space forms of zero sectional curvature, are infranil.  On the other hand, there are many infranil manifolds which are not diffeomorphic to flat manifolds. 

In this paper we generalize and strengthen the Gromov--Ruh Theorem by replacing the $C^0$ bound of $Rm$ with the $L^{n/2}$ bound $\|Rm\|_{n/2}=(\int_M |Rm|^{n/2} dvol)^{2/n}$,  which is weighted by the squared Sobolev constant.

\begin{Def} Let $\varepsilon>0$.  A compact Riemannian manifold $(M, g)$ (or the metric $g$) is called {\it $L^{n/2}$-$\varepsilon$-flat}, provided that $\|Rm\|_{n/2} C_S^2  \leq \varepsilon$, where $C_S$ denotes the Sobolev constant of $(M, g)$ defined in (\ref{2.7}).

A compact manifold $M$ is called {\it $L^{n/2}$-almost flat}, provided that there is a sequence of Riemannian metrics 
$g_k$ on $M$ such that $\|Rm\|_{n/2}(g_k) C_S(g_k)^2  \rightarrow 0$, i.~e.~
$g_k$ is $L^{n/2}$-$\varepsilon_k$-flat, and $\varepsilon_k \rightarrow 0$.
\end{Def}

Our main results are the following three theorems.
\vspace{1ex}

\begin{theoL}\label{theo_A}~
\begin{enumerate}[1)]
\item For each $n \geq 3$ there exists a positive constant $\varepsilon_n$ with the  following property. A compact manifold $M$ of dimension $n$ is diffeomorphic to an infranil manifold if and only it admits a Riemannian metric satisfying 
\begin{equation}  \label{1.1}
\|Rm\|_{n/2} C_S^2 \leq  \varepsilon_n.  
\end{equation}
\item A compact manifold is {\it $L^{n/2}$-almost flat} if and only if it is diffeomorphic to an infranil manifold.
\end{enumerate}
\end{theoL}

Our main tool for proving this theorem is the Ricci flow \cite{Hamilton82}
\begin{equation}  \label{1.2} 
\frac{\partial g}{\partial t}=-2 Ric,
\end{equation}
 which we employ to deform a given metric satisfying \eqref{1.1} into a metric with $diam^2 |K|$ small. The difficulty lies in obtaining this smallness. A Sobolev inequality along the Ricci flow and a diameter estimate for the Ricci flow in \cite{Ye21} play important roles in deriving the needed estimates.

\begin{theoL}\label{theo_B}
  There are positive constants $\varepsilon(n, \gamma)$ and $ c(n, \gamma)$ depending only on $n \geq 3$ and $\gamma>0$ with the following property.  Let $(M, g_0)$ be a compact Riemannian manifold of dimension $n \geq 3$. Assume that $g_0$ satisfies
 $\|Rm\|_{n/2} C_S^2 \leq \varepsilon(n, \gamma)$.  Then the Ricci flow starting at $g_0$ exists on the time interval $[0, T_0]$, where 
\[
T_0 =\gamma\, vol(g(0))^{2/n} C_S(g(0))^2.
\] 
Moreover, the following estimates hold true:
\begin{equation} \label{1.3}
\|Rm\|_{n/2}(g(t)) \leq 2 \|Rm\|_{n/2}(g(0))
\end{equation}
and 
\begin{equation} \label{1.4}
\|Rm\|_{C^0}(g(t))  \leq c(n, \gamma) \frac{C_S(g(0))^2}{t}
  \| Rm \|_{n/2}(g_0)
\end{equation}
for $t \in (0, T_0]$. 
In particular, there holds
\begin{equation} \label{1.5}
\|Rm\|_{C^0}(g(T_0))  \leq \frac{c(n, \gamma)}{\gamma\, vol(g(0))^{2/n}}
  \| Rm \|_{n/2}( g_0).
\end{equation}
Finally, the $vol(g(0))$-rescaled version of the Sobolev inequality \eqref{2.11} with $\alpha=0$ in Theorem \ref{theo_2.2} below holds true for all 
$t \in (0, T_0]$.
\end{theoL}
	
Besides being the basis for Theorem \ref{theo_A}, Theorem \ref{theo_B} also has independent significance from the point view of the theory of the Ricci flow. The next theorem is a corollary of Theorem A, which extends the Gromov--Ruh theorem to the $L^{n/2}$ setting and involves only rudimentary geometric quantities. The Sobolev constant does not appear in the statement of this theorem.

\begin{theoL}\label{theo_C}
There exists a constant $\varepsilon(n, \kappa)>0$ depending only on $n \geq 3$ and $\kappa\geq 0$ such that if $(M, g)$ is a compact Riemannian manifold of dimension $n \geq 3$ with 
$diam^2 \, Ric \geq -\kappa$ and $\|Rm\|_{n/2} (\frac{diam}{vol^{1/n}})^2 \leq \varepsilon(n, \kappa)$, then $M$ is diffeomorphic to an infranil manifold.
\end{theoL}

The $L^p$  version of this theorem with $p>n/2$ is proved in \cite{DPW2000}. The $L^{n/2}$ version is also claimed in \cite{DPW2000}, but the argument in \cite{DPW2000} has a serious gap as pointed out in \cite[Page 6]{Streets16},  namely the time integral of a pointwise curvature estimate there diverges in the $p=n/2$ case because of its $t^{-1}$ order as $t \rightarrow 0$, and hence does not yield the needed bounds in \cite{DPW2000}.  (Though our pointwise curvature estimate \eqref{4.1} (or \eqref{1.4}) is also of $t^{-1}$ order, we only need the time integrals  of $\int_M |\nabla |Rm|^{n/4}|^2$, $\|Rm\|_{n/2}$ and the total scalar curvature to converge, see \eqref{3.9}, \eqref{3.11}, \eqref{4.6} and \eqref{4.7}. )

Gallot's Sobolev constant estimate (Theorem \ref{gallot}) is needed for deriving 
Theorem \ref{theo_C} from Theorem \ref{theo_A}. This estimate has been extended to allow for integral Ricci curvature lower bounds \cite{Gallot1988b, Petersen-Sprouse1998}. Therefore the pointwise condition on the Ricci curvature in Theorem \ref{theo_C} can be replaced by an integral condition as below.

\begin{theoC'}
  There exists a constant $\epsilon (n, p, \kappa, D)>0$ depending only on $n\geq 3$, $p>n/2$, $\kappa\geq 0$, and $D>0$ such that if $(M,g)$ is a compact Riemannian manifold of dimension $n\geq 3$ with $diam_M \le D$, $\left( \frac{1}{\vol M} \int_M (Ric  -\kappa)_-^p\right)^{\frac 1p} \le \epsilon (n, p, \kappa, D)$, and $\left( \frac{1}{\vol M} \int_M |\Rm|^{n/2}\right)^{2/n}   \le \epsilon (n, p, \kappa, D)$, then $M$ is diffeomorphic to an infranil manifold. 
\end{theoC'}

We emphasize that it is an $L^p$ {\it lower} bound for Ricci curvature which is assumed in this theorem.  

Now we discuss the condition $\|Rm\|_{n/2} C_S^2 \leq \varepsilon_n$ in Theorem \ref{theo_A} and the concept of $L^{n/2}$-almost flatness.  Assume the condition $diam^2 |K| $ $\leq \epsilon$ for some $\epsilon \in (0, 1]$. By a rescaling we can assume that $diam=1$ and $|K| \leq \epsilon$. 
Then we have by Gallot's estimate of the Sobolev constant \cite{Gallot1988} (see Theorem~\ref{gallot})
\[
C_S \leq c(n, n-1) vol^{-1/n}.
\]
It follows that
\begin{equation} \label{1.6}
\|Rm\|_{n/2} C_S^2 \leq  c_n \epsilon
\end{equation}
for a positive constant $c_n$ depending only on $n$. Hence 
almost flat manifolds are $L^{n/2}$-almost flat. Moreover, the condition $diam^2 |K| \leq \epsilon_n$ in the Gromov--Ruh Theorem  implies the condition $\|Rm\|_{n/2} C_S^2 \leq \varepsilon_n$ in Theorem A, if we replace $\epsilon_n$ in the Gromov--Ruh Theorem by $c_n^{-1} \varepsilon_n$ if necessary. The above reasoning also shows that the square power of 
$C_S$ in the quantity $\|Rm\|_{n/2} C_S^2$ is natural. Of course, the integral quantity 
$\|Rm \|_{n/2}$ is much weaker than the pointwise quantity $\|Rm\|_{C^0(M)}$ in nature.  The choice of the exponent $n/2$ is also most 
natural because $\|Rm\|_{n/2}$ is scaling invariant. Moreover, $n/2$ is a critical exponent from the point of view of analysis. 

The fact that the Sobolev constant $C_S$  does not stand alone in the conditions in Theorem A and Theorem B is important.  Its square is multiplied to $\|Rm\|_{n/2}$ instead, and the smallness of the product allows $C_S$ to be large.   This is the reason why infranil manifolds which are not diffeomorphic to flat manifolds can occur, and why collapsing can occur under the conditions of Theorem A or Theorem B.  To better understand this, it helps to look at the following theorem, which follows from the estimates in Theorem B.  Its proof and further discussions can be found in \cite{CWY21c}.
\begin{Thm}[\cite{CWY21c}]\label{thm_1.1}
    For each $n \geq 3$ and each $C>0$ there is a positive constant $\varepsilon(n, C)$ depending only on 
$n$ and $C$ with the following property.  An arbitrary compact Riemannian manifold $(M, g)$ of dimension $n \geq 3$ with $\|Rm\|_{n/2} \leq \varepsilon(n, C)$ 
and $C_S \leq C$ is diffeomorphic to a flat space form.
\end{Thm}
 
 When $n=4$ this has previously been established by Streets in his study of the gradient flow of $\int_M |\Rm|^2\ dV$ \cite[Corollary 1.17]{Streets16}.   
  
One can also ask whether in Theorem \ref{theo_A} the quantity $\|Rm\|_{n/2}$ alone, without the weight $C_S^2$, would be sufficient for deriving the infranil conclusion.  The answer is no. To see this, consider the manifold $S^{n-1}\times S^1$ for $n \geq 3$ equipped with the metric $g_{\epsilon}=g_{S^{n-1}}+\epsilon^2 g_{S^1}$, where $g_{S^{n-1}}$ and $g_{S^1}$ are the standard metrics of $S^{n-1}\subset\mathbb{R}^n$ and $S^1\subset\mathbb{R}^2$ respectively. We have $\|Rm_{g_\epsilon}\|_{n/2}=c(n)\epsilon^{2/n}$ for some positive constant $c(n)$ depending only on $n$, and this tends to zero as $\epsilon$ goes to zero. However
the universal cover of $S^{n-1}\times S^1$ is not diffeomorphic to $\mathbb{R}^n$, and therefore it is not diffeomorphic to an infranil manifold. 

We can also apply Gallot's Sobolev constant estimate in this example to deduce that $C_S(g_{\epsilon}) \leq \tilde c(n) \epsilon^{-1/n}$
for a positive constant $\tilde c(n)$. 
As a result we conclude further that in Theorem \ref{theo_A} (and in Theorem \ref{theo_B}) the squared power of $C_S$ in $\|Rm\|_{n/2} C_S^2$ is sharp. Namely for any $\alpha<2$ and $\delta>0$ there exists a  compact manifold with $\|Rm\|_{n/2}C_S^{\alpha}<\delta$ which is not diffeomorphic to an infranil manifold. (Note that for $\alpha_1\geq\alpha_2$, $\|Rm\|_{n/2}C_S^{\alpha_1}$ small implies $\|Rm\|_{n/2}C_S^{\alpha_2}$ small.)

Finally we observe that the condition $diam^2 \sup_M  |Rm| \leq \varepsilon$ (which is equivalent to $diam^2 \sup_M |K| \leq c(n) \varepsilon$ for a positive constant $c(n)$) can be viewed as a pinching condition in the form
\begin{equation}
\sup_M |Rm-\alpha G| \leq \varepsilon diam^{-2}
\end{equation}
for the case $\alpha=0$, where $G(x, y, z, w)=g(x, z)g(y, w)-g(x, w)g(y, z)$.  Similarly, the condition $\|Rm\|_{n/2} C_S^2 \leq \varepsilon$ can be viewed as the $\alpha=0$ case of the 
pinching condition
\begin{equation} \label{1.7}
\|Rm-\alpha G\|_{n/2} \leq \varepsilon C_S^{-2}.
\end{equation}
We  have also obtained curvature estimates along the lines of Theorem B  for the cases of positive or negative $\alpha$; see \cite{CWY21c}, where besides Theorem \ref{thm_1.1}, pinching theorems for the $\alpha>0$ and $\alpha<0$ cases are also proved (with $\alpha$ being given by the average scalar curvature multiplied by a dimensional constant), in which the condition $\|Rm-\alpha G\|_{n/2} C_S^2 \leq \varepsilon_n$ and some additional conditions are assumed to lead to a space form conclusion.

The research in this paper was a project parallel to \cite{CWY2021}, which studied integral curvature pinching of positive Yamabe metrics, and, like 
that work, is related to \cite{Chen}, which studied integral curvature pinching on asymptotically flat manifolds.

\section{Preliminaries} 

The following evolution equation for Riemann curvature tensor $Rm$ holds along the Ricci flow \eqref{1.2} on a manifold $M$ \cite{Hamilton86}
\begin{equation} \label{2.1}
\frac{\partial Rm}{\partial t} = \Delta Rm + Rm^2+Rm^{\sharp}+ Ric* Rm,
\end{equation}
where $Ric*Rm$ is a quadratic expression involving $Ric$ and $Rm$.
It follows that 
\begin{equation} \label{2.2}
\frac{\partial |Rm|}{\partial t} \leq \Delta |Rm| +c(n) |Rm|^2
\end{equation}
in the sense of distributions, or in the weak sense (see for instance \cite[Equation (6.1)]{Chow}). Here and in the sequel, $c(n)$ denotes a positive constant depending only on $n$, whose value can be different in different places. (The $c(n)$s appearing in different lines of the same computation denote the same constant.)   Assume that $M$ is compact. Then it follows from \eqref{2.1} or \eqref{2.2} that 
\begin{equation} \label{2.3}
\frac{\partial }{\partial t} \int_M |Rm|^p \leq c(n)p \int_M |Rm|^{p+1} -\frac{4(p-1)}{p} \int_M |\nabla |Rm|^{p/2}|^2
\end{equation}
for $p\geq 1$ and a.e.~$t$. Here and in the sequel the notation of the volume form is omitted.  (To prove \eqref{2.3}, one first calculates in terms of $|Rm|_{\epsilon}=(|Rm|^2+\epsilon^2)^{1/2}$ and then lets $\epsilon \rightarrow 0$, applying Fatou's lemma.) 
In particular we have 
\begin{equation} \label{2.4}
\frac{\partial }{\partial t} \int_M |Rm|^{n/2} \leq c(n) \int_M |Rm|^{\frac{n}{2}+1} -\frac{4(n-2)}{n} \int_M |\nabla |Rm|_{}^{n/2}|^2.
\end{equation}
On the other hand, we have by the H\"{o}lder inequality
\begin{equation} \label{2.5}
\int_M |Rm|^{p+1} \leq (\int_M |Rm|^{n/2})^{2/n} (\int_M |Rm|^{p \cdot
 \frac{n}{n-2}})^{\frac{n-2}{n}}.
\end{equation} 
In particular
\begin{equation} \label{2.6}
\int_M |Rm|_{}^{\frac{n}{2}+1} \leq (\int_M |Rm|_{}^{n/2})^{2/n} (\int_M |Rm|_{}^{\frac{n}{2} \cdot
 \frac{n}{n-2}})^{\frac{n-2}{n}}.
\end{equation}

Next we recall the definition of the Sobolev constant of a compact Riemannian manifold.

\begin{Def}[Sobolev constant]
Let $(M,g)$ be a compact Riemannian manifold of dimension $n\geq 3$. Its {\it ($L^2$) Sobolev constant} $C_S(M,g)$ is defined to be
\begin{equation} C_S(M,g)=\sup\left\{\|u\|_{\frac{2n}{n-2}}-\frac{1}{\vol(M,g)^{1/n}}\|u\|_2:~u\in C^1(M),\|\nabla u\|_2\leq 1\right\}.\label{2.7} 
\end{equation}
Equivalently, $C_S(M,g)$ is the smallest number $C$ such that the Sobolev inequality
\[\|u\|_{\frac{2n}{n-2}}\leq C\|\nabla u\|_{2}+\frac{1}{\vol(M,g)^{1/n}}\|u\|_2\]
holds true for all $u\in C^1(M)$.  (Note that we can replace $C^1(M)$ by the Sobolev space 
$W^{1,2}(M)$.)
\end{Def}

Gallot's following estimate for the Sobolev constant is well-known.

\begin{Thm}[\cite{Gallot1988}] \label{gallot}
For each $n \geq 3$ and each $\kappa \geq 0$ there is a constant  $c(n, \kappa)>0$ with the following property. Let $(M^n,g)$ be a compact Riemannian manifold of dimension $n \geq 3 $  satisfying $diam^2 \, Ric \geq -\kappa$ for some constant $\kappa \geq 0$. Then its Sobolev constant satisfies $C_S \le c(n, \kappa) \frac{diam}{vol^{1/n}}$.
\end{Thm}

We will apply the Sobolev inequality along the Ricci flow in the following theorem from \cite{Ye21}. (We only need the version of this theorem for compact manifolds.) Like the Sobolev inequality in \cite{Ye15}, it is based on the monotonicity of Perelman's entropy functional \cite{P} and
harmonic analysis of the heat operator. 

\begin{Thm}[\cite{Ye21}]\label{theo_2.2}
Consider a smooth solution of the Ricci flow $g=g(t)$ on a compact manifold $M$ of dimension $n \geq 3$, $t\in [0, T)$ for some $T>0$.  Let $\alpha$ be a constant. Assume 
\begin{equation} \label{2.8}
\|(R-\alpha)^-\|_{n/2}(0) C_S^2(0) \leq 1.
\end{equation} 
 ($0$ indicates the metric $g(0)$. Similar notations will be used for $t$.)  Set
 \begin{equation} \label{2.9} 
\delta_0=C_S^{-2}(0)+\|(R-\alpha)^-\|_{n/2}(0).
\end{equation}
Assume that $t\in[0, T)$ satisfies 
\begin{equation} \label{2.10}
a_n \|(R-\alpha)^+\|_{n/2}(t) C_S^2(0) e^{\frac{2t}{n}(4\delta_0-\alpha)} \leq 1
\end{equation}
for  a suitable constant $a_n\geq 1$ depending only on $n$.  For convenience of presentation we also assume that $vol(g(0))=1$. Then there holds
\begin{equation} \label{2.11}
(\int_M |u|^{\frac{2n}{n-2}})^{\frac{n-2}{n}} \leq c(n) e^{\frac{2t}{n}(4\delta_0-\alpha)} (C_S^2(0)\int_M |\nabla u|^2 +\int_M u^2)
\end{equation}
at time $t$ for all $u \in W^{1,2}(M)$ and a suitable positive constant $c(n)$.
In general, the $vol(g(0))$-rescaled version of \eqref{2.11} holds true without 
the condition $vol(g(0))=1$.
\end{Thm}

\begin{Cor}\label{cor_2.1}
Consider the set-up of Theorem \ref{theo_2.2}. (In particular we assume $vol(g(0))=1$.)  Assume 
\begin{equation} \label{2.12}
 \|Rm\|_{n/2}(0) C_S^2(0) \leq \frac{1}{n(n-1)}
\end{equation}
instead of \eqref{2.8} and, at time $t$,
\begin{equation} \label{2.13}
a_n \|Rm\|_{n/2}(t) C_S^2(0) e^{\frac{8 \delta_0 t}{n}} \leq \frac{1}{n(n-1)}
\end{equation}
instead of \eqref{2.10}, where  $\alpha=0$ in the definition \eqref{2.9} of $\delta_0$. Then \eqref{2.11} with $\alpha=0$ holds true at
 time $t$ for all $u \in W^{1,2}(M)$.
\end{Cor}

Except in Theorem \ref{thm_4.1}, henceforth $\alpha$ is chosen to be zero in the definition of $\delta_0$. 

\section{Some Integral Estimates for $Rm$}  

\subsection{}

Now we consider a smooth solution $g=g(t), t \in [0, T)$ of the Ricci flow on a compact manifold $M$ of dimension $n \geq 3$. Assume $vol(0)=1$ and \eqref{2.12}.  For a time $t\in (0, T)$ satisfying \eqref{2.13} we apply \eqref{2.11} to
 $u=|Rm|^{p/2}$ to deduce 
\begin{equation} \label{3.1}
(\int_M |Rm|_{}^{p \cdot \frac{n}{n-2}})^{\frac{n-2}{n}}\leq c(n) e^{\frac{8t}{n}\delta_0} (C_S^2(0) \int_M |\nabla 
|Rm|_{}^{p/2}|^2 +\int |Rm|_{}^{p})  
\end{equation}
at $t$ for $p\geq 1$. In particular we have 
\begin{equation} \label{3.2}
(\int_M |Rm|_{}^{\frac{n}{2} \cdot \frac{n}{n-2}})^{\frac{n-2}{n}}\leq c(n) e^{\frac{8t}{n}\delta_0} (C_S^2(0)\int_M |\nabla |Rm|_{}^{n/4}|^2 +\int |Rm|_{}^{n/2})
\end{equation}
at $t$. Employing \eqref{2.5} and \eqref{3.1} we then deduce
\begin{equation} \label{3.3}
\int_M |Rm|_{}^{p+1} \leq c(n)e^{\frac{8t}{n}\delta_0}\|Rm\|_{n/2} (C_S^2(0) \int_M |\nabla |Rm|^{p/2}|^2+\int_M |Rm|_{}^{p})
\end{equation}
at $t$.  In particular we have 
\begin{equation} \label{3.4}
\int_M |Rm|_{}^{\frac{n}{2}+1} \leq c(n)e^{\frac{8t}{n}\delta_0}\|Rm\|_{n/2} (C_S^2(0) \int_M |\nabla |Rm|_{}^{n/4}|^2+\int_M |Rm|_{}^{n/2}).
\end{equation} 
at $t$.
\subsection{}  

 Set 
\begin{equation} \label{3.5}
J(t)=\left.\int_M |Rm|_{}^{n/2} \right|_t,\quad
\theta_t=\|Rm\|_{n/2}(t) C_S^2(0),\quad
 \chi(t)= c(n) e^{\frac{8t}{n}\delta_0} \| Rm\|_{n/2}(t),
\end{equation}
where $c(n)$ is from \eqref{2.11}.
 Now we consider a time $t\in (0, T)$ which satisfies the following condition
\begin{equation} \label{3.6}
c(n) e^{\frac{8t}{n}\delta_0} \|Rm\|_{n/2}(t) C_S^2(0)\leq \frac{1}{n(n-1)},  
\end{equation} 
where $c(n)=\max\{a_n,  \tilde c(n)\}$ with $a_n$ from \eqref{2.10} and $\tilde c(n)$ standing for the $c(n)$ in \eqref{2.11}.  We deduce from \eqref{2.4} and \eqref{3.4}
\begin{equation} \label{3.7}
\frac{\partial }{\partial t} \int_M |Rm|_{}^{n/2} \leq  -\frac{3(n-2)}{n} \int_M |\nabla |Rm|_{}^{n/4}|^2+c(n)e^{\frac{8t}{n}\delta_0}  
\|Rm\|_{n/2}
\int_M |Rm|_{}^{n/2},
\end{equation}
i.e.
\begin{equation} \label{3.8}
J'(t)\leq  -\frac{3(n-2)}{n} \int_M |\nabla |Rm|_{}^{n/4}|^2+ c(n)e^{\frac{8t}{n}\delta_0}  J(t)^{1+\frac{2}{n}}.
\end{equation}

Next we consider $t>0$ such that for all $0<s\leq t$, \eqref{3.6} holds true with $t$ replaced by $s$. 
Then we can integrate \eqref{3.8} to deduce 
\begin{equation} \label{3.9}
 J(t) e^{-\int_0^t \chi}+\frac{3(n-2)}{n}\int_0^t e^{-\int_0^s \chi}\int_M |\nabla |Rm|_{}^{n/4}|^2 \leq 
J(0).
\end{equation} 

One simple consequence of \eqref{3.9} is the following estimate
\begin{equation} \label{3.10}
J(t) \leq e^{\int_0^t \chi} J(0) 
\end{equation}
There holds 
\begin{equation} \label{3.11}
\int_0^t \chi \leq c(n) e^{\frac{8t}{n}\delta_0} \int_0^t \|Rm\|_{n/2}.
\end{equation}
To proceed, we consider $t>0$ such that \eqref{3.6} with $t$ replaced by $s$ and the inequality
\begin{equation} \label{3.12}
J(s) \leq 2^{n/2} J(0)
\end{equation}
holds true for all $s \in [0, t]$. There holds
\[
\int_0^t \|Rm\|_{n/2} \leq 2 t\|Rm\|_{n/2}(0).
\]
Hence we deduce
\begin{equation} \label{3.13}
J(t)\leq exp(2 c(n)e^{\frac{8t}{n}\delta_0}t \|Rm\|_{n/2}(0)) J(0).
\end{equation}
Alternatively, we can drop the first term on the right hand side of \eqref{3.8} and then integrate it to deduce an estimate which can be used instead of \eqref{3.9}.

We formulate the above estimate as a lemma.

\begin{Lem}\label{lem_3.1}
Let $g=g(t)$ be a smooth solution of the Ricci flow on a compact manifold $M$ of dimension $n \geq 3$ and the time interval $[0, T)$ for some $T>0$. Assume \eqref{2.12} and $vol(g(0))=1$. Let $t\in [0,T)$. Assume that for each $s\in[0, t]$, the inequality \eqref{3.6} with $t$ replaced by $s$ and the inequality \eqref{3.12} hold true. Then the estimate \eqref{3.13} holds true.
\end{Lem}

The next lemma follows from Lemma \ref{lem_3.1}.

\begin{Lem}\label{lem_3.2}
For each $n \geq 3$ and each $\gamma>0$, there exists a positive constant $\varepsilon(n, \gamma)$ depending only on $n$ and $\gamma$ with the following property.  Let $g=g(t)$ be a smooth solution of the Ricci flow on a compact manifold $M$ of dimension $n \geq 3$, $ t\in [0, T)$ for some $T>0$, such that 
\begin{equation} \label{3.14}
\|Rm\|_{n/2}(0) C_S^2(0) \leq \varepsilon(n, \gamma).
\end{equation}
Set 
\begin{equation} \label{3.15}
T_0=\gamma vol(0)^{2/n} C_S^2(0).
\end{equation}
Then there holds
\begin{equation} \label{3.16}
\|Rm\|_{n/2}(t) \leq 2 \|Rm\|_{n/2}(0)
\end{equation}
for all $t \in [0, T_0] \cap [0, T)$.
\end{Lem}

Note that one can choose $\gamma=1$ in this lemma.

\begin{proof}
Let $n\geq 3$ and $\gamma>0$ be given. 
We first define 
\begin{equation} \label {3.17}
T_1=vol(0)^{2/n}\min \{\gamma C_S^2(0), \frac{c(n, \gamma)}{\|Rm\|_{n/2}(0)}\},
\end{equation}
where $c(n, \gamma)$ is to be determined.  Set
\begin{equation} \label{3.18}
\varepsilon(n, \gamma)=\min\{\frac{n-2}{nc(n)}, \frac{1}{n(n-1)}, b(n, \gamma)\},
\end{equation} 
where $c(n)$ is from \eqref{3.6} and $b(n, \gamma)$ is to be defined. 

Consider a solution $g=g(t)$ of the Ricci flow as stated in the theorem, which in particular satisfies \eqref{3.14}. By a rescaling we can assume  $vol(0)=1$. The case $J(0)=0$ is trivial. So we assume $J(0)>0$.  Define
$I=\{t\in [0, T_1) \cap [0, T):
 J(s) \leq 2^n J(0),  c(n) e^{\frac{8t}{n}\delta_0} \|Rm\|_{n/2}(t) C_S^2(0)\leq 
\frac{1}{n(n-1)}
\mbox{ for all } s \in [0, t]\}, $
where $c(n)$ is again from \eqref{3.6}. 
Then $I$ is closed in $[0, T_1) \cap [0, T)$. It is readily checked that $0 \in I$.  
Let $t \in I$. There holds
\begin{equation} \label{3.19}
t\delta_0<T_1(C_S^{-2}(0)+n(n-1) \|Rm\|_{n/2}(0))\leq \gamma+n(n-1) c(n, \gamma).
\end{equation}
By Lemma \ref{lem_3.1} we infer
\begin{align} \label{3.20}
J(t)&\leq exp\left(2c(n)e^{\frac{8t}{n}\delta_0}t \|Rm\|_{n/2}(0)\right) J(0)
\\
&< 
exp\left (2c(n)e^{\frac{8}{n}(\gamma+n(n-1) c(n, \gamma))} c(n, \gamma)\right)J(0).\notag
\end{align}
We define  $c(n, \gamma)$ to be the unique solution of the equation
\begin{equation} \label{3.21}
exp\left (2c(n)e^{\frac{8}{n}(\gamma+n(n-1) x)}x\right)=2^{n}.
\end{equation}
Then we deduce 
\begin{equation} \label{3.22} 
J(t)< 2^{n} J(0).
\end{equation}
 It then also follows that
\[
c(n)e^{\frac{8t}{n}\delta_0} \|Rm\|_{n/2}(t)C_S^2(0) < 2c(n)e^{\frac{8}{n}(\gamma+n(n-1)c(n, \gamma))} \|Rm\|_{n/2}(0) C_S^2(0).
\]
We set
\begin{equation} \label{3.23}
b(n, \gamma)=\min\left\{\frac{1}{2n(n-1)c(n)} e^{-\frac{8}{n}(\gamma+n(n-1)c(n, \gamma))}, \frac{c(n, \gamma)}{\gamma}\right\}
\end{equation}
and deduce 
\begin{equation}
c(n)e^{\frac{8t}{n}\delta_0} \|Rm\|_{n/2}(t)C_S^2(0)<\frac{1}{n(n-1)}.
\end{equation}
It follows that $I$ is open in $[0, T_1)\cap [0, T)$. Consequently, $I=[0, T_1) \cap [0,T)$, and hence $J(t) < 2^n J(0)$ for all $t \in [0, T_1)\cap [0, T)$. Finally, by \eqref{3.14} and the definitions \eqref{3.18} and \eqref{3.23} we infer 
\[
\frac{c(n, \gamma)}{\|Rm\|_{n/2}(0)}\geq \gamma C_S^2(0),
\]
and hence $T_1=T_0$.  It follows that the desired estimate \eqref{3.16} holds true for all $t\in [0, T_0) \cap [0, T)$. By continuity, it also holds true for all $t \in [0, T_0] \cap [0, T]$.
\end{proof}

\subsection{} 

For $n \geq 3$ set $p_0=p_0(n)= \frac{n^2}{n-2}$.

\begin{Lem}\label{lem_n2_control} Assume the same as in Lemma \ref{lem_3.2}. In addition, assume that $vol(0)=1$. For each $t\in (0, T_0] \cap (0, T)$ there exists a $t^*\in [t/3, t/2]$ such that 
\begin{equation} \label{3.24}
\|Rm\|_{p_0/2}(t^*) \leq c(n, \gamma)  \|Rm\|_{n/2}(0)\left(\frac{C_S^2(0)}{t}+1\right)^{2/n}
\end{equation}
for a suitable positive constant $c(n, \gamma)$ depending only on $n$ and $\gamma$. 
\end{Lem}

The estimate \eqref{3.24} needs to be modified by a factor if the condition $vol(0)$ is dropped. This remark also applies to a number of estimates in the sequel.

\begin{proof}[Proof of Lemma \ref{lem_n2_control}]
By \eqref{3.9} we deduce for $t \in (0, T_0]\cap [0, T)$
\begin{equation} \label{3.25}
J(t)+ \frac{3(n-2)}{n} e^{\int_{t/2}^t \chi }\int_{t/3}^{t/2} \int_M |\nabla |Rm|_{}^{n/4}|^2 \leq e^{\int_0^t \chi} 
J(0).
\end{equation}
Consequently there holds 
\begin{equation} \label{3.26}
\frac{3(n-2)}{n} \int_{t/3}^{t/2} \int_M |\nabla |Rm|_{}^{n/4}|^2 \leq e^{\int_0^{t/2} \chi} J(0).
\end{equation}
Hence there exists a $t^*\in [t/3, t/2]$ such that 
\[
\int_M |\nabla |Rm|^{n/4}|^2 |_{t^*} \leq \frac{2n}{n-2} \cdot \frac{1}{t}  e^{\int_0^{t/2} \chi} J(0).
\]
Since $t\leq T_0$, applying the estimates \eqref{3.14}--\eqref{3.16} we infer 
\begin{equation} \label{3.27}
\int_M |\nabla |Rm|^{n/4}|^2 |_{t^*}\leq \frac{c(n, \gamma)}{t} J(0)
\end{equation}
for a new $c(n, \gamma)>0$. Applying \eqref{3.2} and \eqref{3.16} we deduce
\begin{align}
(\int_M |Rm|_{}^{\frac{n}{2} \cdot \frac{n}{n-2}})^{\frac{n-2}{n}}|_{t^*} &\leq c(n) e^{\frac{8t}{n}\delta_0} (C_S^2(0)\frac{c(n, \gamma)}{t} J(0) +\int |Rm|_{}^{n/2})
\notag 
\\
&\leq c'(n, \gamma)\int_M |Rm|^{n/2}|_0 \left(\frac{C_S^2(0)}{t}+1\right) \notag
\end{align}
for a constant $c'(n, \gamma)>0$, which yields \eqref{3.24}.
\end{proof}

\begin{Lem}\label{lem_3.4}
Assume the same as in Lemma \ref{lem_3.2}, except that the constant 
$\varepsilon(n, \gamma)$ in \eqref{3.14} is replaced by $\varepsilon_1(n, \gamma)$ given in \eqref{3.29} below. In addition, assume $vol(0)=1$.Let $t \in (0, T_0] \cap (0, T)$. Then there holds 
\begin{equation}
\|Rm\|_{p_0/2} \leq c(n, \gamma) \|Rm\|_{n/2} (0)(\frac{C_S^2(0)}{t}+1)^{2/n}
\end{equation}
on $[t/2, t]$ for a constant $c(n, \gamma)>0$.
\end{Lem}

\begin{proof}
We set $q_0=p_0/2$.  
By \eqref{3.3} and Lemma \ref{lem_3.2} we infer
\begin{align}
\int_M |Rm|^{q_0+1} 
&\leq c(n) e^{\frac{8t}{n}\delta_0}\|Rm\|_{n/2} (C_S^2(0) \int_M |\nabla |Rm|^{q_0/2}|^2+\int_M |Rm|^{q_0}) \notag \\
&\leq c_3(n, \gamma) \frac{q_0-1}{q_0} (\|Rm\|_{n/2}(0)C_S(0)^2\int_M |\nabla |Rm|^{q_0/2}|^2
\notag\\
&+\|Rm\|_{n/2}(0) \int_M |Rm|^{q_0})\notag
\end{align}
for a constant $c_3(n, \gamma)>0$. Set
\begin{equation}\label{3.29}
\varepsilon_1(n, \gamma)=\min\{\varepsilon(n, \gamma), c_3(n, \gamma)^{-1}\}.
\end{equation} 
Since it is assumed that $\|Rm\|_{n/2}(0)C_S(0)^2\leq \varepsilon_1(n, \gamma)$, 
We deduce
\begin{equation} \label{3.30}
\int_M |Rm|^{q_0+1} \leq \frac{q_0-1}{q_0} (\int_M |\nabla |Rm|^{q_0/2}|^2+c_3(n, \gamma)\|Rm\|_{n/2}(0) \int_M |Rm|^{q_0}).
\end{equation}
Applying \eqref{2.2} with $p=q_0$ we then deduce \vspace{2mm}
\begin{equation} \label{3.31}
\frac{\partial }{\partial t} \int_M |Rm|^{q_0} \leq  -\frac{3(q_0-1)}{q_0} \int_M |\nabla |Rm|^{q_0/2}|^2+
c_4(n, \gamma)\|Rm\|_{n/2}(0) \int_M |Rm|^{q_0}
\end{equation}
for a constant $c_4(n,\gamma)>0$. 

The desired estimate follows from an integration of \eqref{3.31} and  Lemma \ref{lem_n2_control}. 
\end{proof}

\begin{Cor}
Under the assumption of Lemma \ref{lem_3.4} there holds
\[
\|Rm\|_{p_0/2} \leq c(n, \gamma)\|Rm\|_{n/2}(0).
\]
in the interval $[T_0/2, T_0] \cap [0, T)$ for a constant $c(n, \gamma)>0$.
\end{Cor}

\section{Proofs of the Main Theorems}

\begin{Lem}\label{lem_4.1}
Under the assumption of Lemma \ref{lem_3.4} there holds for each $t\in(0, T_0] \cap (0, T)$
\begin{equation} \label{4.1}
\|Rm\|_{C^0}(t) \leq c(n, \gamma) 
 \frac{C_S(0)^2}{t} \| Rm \|_{n/2}(0)
\end{equation}
with a constant $c(n, \gamma)>0$.
\end{Lem}

The proof of this lemma is based on the technique of the well-known Moser iteration and presented in the appendix.

\begin{proof}[Proof of Theorem \ref{theo_B}]
The smooth solution $g=g(t)$ of the Ricci flow on $M$ with $g(0)= g_0$ exists on 
a maximal time interval $[0, T)$ for some $T>0$. ($T=\infty$ is allowed.) 
 We claim that $T > T_0$. Assume $T \leq T_0$.
Then we have 
\[
\limsup_{t \rightarrow T} \|Rm\|_{C^0}(t)=\infty,
\]
 contradicting Lemma \ref{lem_4.1}. The estimate \eqref{1.4} in Theorem \ref{theo_B} 
follows from Lemma \ref{lem_4.1}. The estimate \eqref{1.5} is a special case of \eqref{1.4}. The estimate \eqref{1.3} follows from Lemma \ref{lem_3.2}.  

Finally we observe that the condition $\|Rm\|_{n/2} C_S^2 \leq \varepsilon(n,\gamma)$ for $g_0$ implies the condition \eqref{2.12}, and the proof of Lemma \ref{lem_3.2} implies the condition \eqref{2.13} for all $t\in (0, T_0]$.
Hence we can apply Corollary \ref{cor_2.1} to deduce the Sobolev inequality \eqref{2.11} for all $t\in (0, T_0]$.
\end{proof}

To prove Theorem \ref{theo_A}, we will apply the following diameter estimate theorem from \cite{Ye21}.
\vspace{1ex}

\begin{Thm}[\cite{Ye21}]\label{thm_4.1}~
\begin{enumerate}[1)]
\item Consider a compact Riemannian manifold $(M, g)$ of dimension $n \geq 3$.  Assume \begin{equation} \label{4.2}
\|u\|^2_{\frac{2n}{n-2}} \leq A\int_M |\nabla u|^2+ \frac{B}{vol(M, g)^{2/n}}\int_M u^2
\end{equation}
for all u $\in W^{1,2}(M)$, with positive constants $A$ and $B$.  Then there holds 
\begin{equation} \label{4.3}
\frac{diam(M, g)}{vol(M, g)^{1/n}}\leq   2^{\frac{n}{2}+1}(2^{\frac{n}{2}}B^{\frac{n}{2}}+1)\sqrt{\frac{A}{B}}.
\end{equation}
\item Let $g=g(t), t\in [0, T)$ be a smooth solution of the Ricci flow on a compact manifold 
$M$ of dimension $n \geq 3$ for a $T>0$, such that $vol(g(0))=1$ and \eqref{2.8} holds true for some constant $\alpha$ (or \eqref{2.12} holds true in the case 
$\alpha=0$). Then there exists a positive constant $c(n)$ such that 
\begin{equation} \label{4.4}
diam(M, g(t))\leq c(n) \left(e^{\frac{t}{n}(4\delta_0-\alpha)}vol(M, g(t))+1 \right)C_S(g(0))
\end{equation}
for $t\in (0, T)$ whenever the condition \eqref{2.10} holds true (or the condition \eqref{2.13} holds true in the case $\alpha=0$).
\end{enumerate}
\end{Thm}

\begin{proof}[Proof of Theorem \ref{theo_A}] 1) We assume $\varepsilon_n \leq \varepsilon(n, 1)$ and will determine its value below, where $\varepsilon(n, \gamma)$ with $\gamma=1$ is from Theorem \ref{theo_B}.  Let $(M, g_0)$ be a compact Riemannian manifold satisfying $\|Rm\|_{n/2} C_S^2 \leq \varepsilon_n$. By a rescaling we can assume that $vol(g_0)=1$. 
Consider the smooth solution $g=g(t)$ of the Ricci flow on $M$ with $g(0)=g_0$. By Theorem \ref{theo_B}, $g(t)$ exists on $[0, T_0]$ with $T_0=C_S(g_0)^2$ and satisfies
\begin{equation} \label{4.5}
\|Rm\|_{C^0}(g(T_0))\leq c(n, 1) \|Rm\|_{n/2}(g_0).
\end{equation}

There holds 
\begin{equation} \label{4.6}
\frac{dvol}{d t}= -\int_M R.
\end{equation}
Hence
\begin{equation} \label{4.7}
|\frac{dvol}{dt}| \leq  \|Rm\|_{n/2} vol^{\frac{n-2}{n}}
\end{equation}
and then $|\frac{dvol^{2/n}}{dt}| \leq \frac{2}{n} \|Rm\|_{n/2}$. Applying the estimate 
\eqref{1.3} we then deduce 
\[
vol(g(T_0))^{2/n}\leq 1+\frac{4}{n}\|Rm\|_{n/2} (g_0) C_S(g_0)^2\leq 1+\frac{4}{n}\varepsilon_n.
\]
By \eqref{3.18} there holds $\varepsilon_n \leq \frac{1}{n(n-1)}$, whence
\begin{equation} \label{4.8}
vol(g(T_0))\leq \left(1+\frac{4}{n^2(n-1)} \right)^{\frac{n}{2}}.
\end{equation}
Applying this estimate, Theorem \ref{thm_4.1}, 1) and the argument at the end of the proof of Theorem \ref{theo_B}
(alternatively, applying Theorem \ref{theo_B} and Theorem \ref{thm_4.1}, 2)), as well as the condition $\|Rm\|_{n/2}(g_0) C_S(g_0)^2 \leq \varepsilon_n$, we then deduce 
\begin{equation} \label{4.9}
diam(g(T_0)) \leq c(n) C_S(g_0) 
\end{equation}
for a positive constant $c(n)$. Combining this estimate with \eqref{4.5} we infer\vspace{1mm}
\begin{align} 
\|Rm\|_{C^0}(g(T_0))diam(g(T_0))^2 &\leq c(n, 1)c(n) \|Rm\|_{n/2}(g_0)C_S(g_0)^2 \label{4.10}
\\
&\leq c(n, 1) c(n) \varepsilon_n.\notag
\end{align}
Set 
\begin{equation} \label{4.11}
\varepsilon_n=\min\{\varepsilon(n, 1), \frac{\epsilon_n}{c(n, 1) c(n)}\},
\end{equation}
where $\epsilon_n$ is from the Gromov--Ruh Theorem. Then we can apply the Gromov--Ruh Theorem to conclude that 
$M$ is diffeomorphic to an infranil manifold. (Note that $|K| \leq |Rm|$.) 
 
Conversely, let $M$ be diffeomorphic to an infranil manifold of dimension $n\geq 3$. By \cite{Gromov1978a}, $M$ is almost flat. Hence $M$ is $L^{n/2}$-almost flat, as shown by the argument around \eqref{1.6} in the Introduction. We infer that $M$ admits metrics satisfying the condition $\|Rm\|_{n/2}C_S^2 \leq \varepsilon_n$.

2) Manifolds which are diffeomorphic to infranil manifolds are almost flat by the Gromov--Ruh Theorem. As mentioned above, almost flat manifolds are 
$L^{n/2}$-almost flat. Hence manifolds which are diffeomorphic to infranil manifolds are $L^{n/2}$-almost flat.  On the other hand, it follows from part 1) of Theorem \ref{theo_A} that $L^{n/2}$-almost flat manifolds are diffeomorphic to infranil mnaifolds.   
\end{proof}

\begin{proof}[Proof of Theorem \ref{theo_C}]We set $\epsilon(n, \kappa)=c(n, \kappa)^{-2}\varepsilon_n$ for 
$n \geq 3$ and $\kappa \geq 0$, where $c(n, \kappa)$ is from Theorem \ref{gallot}.  Let $(M, g)$ be a compact Riemannian manifold of dimension $n \geq 3$ satisfying $diam^2 Ric \geq -\kappa$ and $\|Rm\|_{n/2} (\frac{diam}{vol^{1/n}})^2\leq  \epsilon(n, \kappa)$ for some $\kappa \geq 0$.
 By Theorem \ref{gallot}, $(M, g)$ then satisfies 
 $C_S \leq c(n, \kappa) \frac{diam}{vol^{1/n}}$. Hence we deduce 
\[
\|Rm\|_{n/2} C_S^2 \leq c(n, \kappa)^2 \|Rm\|_{n/2} (\frac{diam}{vol^{1/n}})^2\leq c(n, \kappa)^2 \epsilon(n, \kappa)=\varepsilon_n.
\] 
Thus we can apply Theorem \ref{theo_A} to conclude that $M$ is diffeomorphic to an infranil manifold.
\end{proof}

\appendix
\section{ $C^0$ Estimates for $Rm$ } 

The purpose of this section is to present the proof of Lemma \ref{lem_4.1}, which is based on Moser iteration \cite{Moser}. For the purpose of carefully verifying all the details and working out the explicit constants in the estimates, and for the convenience  of the reader, we present a  detailed and self-contained proof, using the same notations we used previously, and in particular the definitions in \eqref{3.6}. In principle, one can also apply for example the Moser type estimates in \cite[Theorem 4]{Yang}.

Let $\epsilon>0$.  We have for a nonnegative measurable function $f$ on a Riemannian manifold $(M, g)$ of dimension $n\geq 3$
\begin{align} \label{5.1} 
\left(\int_M f^{\frac{p_0}{p_0-2}}\right)^{\frac{p_0-2}{p_0}} &\leq \left(\int_M f^{\frac{n}{n-2}}\right)^{\frac{n-2}{p_0}} 
\left(\int_M f\right)^{\frac{p_0-2}{p_0} \cdot\frac{p_0-n}{p_0-2}} 
\\
&=\left(\varepsilon^{\frac{p_0-n}{p_0}}\int_M f^{\frac{n}{n-2}}\right)^{\frac{n-2}{p_0}} 
\left(\varepsilon^{-\frac{n-2}{p_0}}\int_M f\right)^{\frac{p_0-n}{p_0}} \notag
\\
&\leq \frac{p_0-n}{p_0} \varepsilon^{-\frac{n-2}{p_0}} \int_M f+\frac{n}{p_0} \varepsilon^{\frac{p_0-n}{p_0}\cdot \frac{n-2}{p_0} \cdot \frac{p_0}{n}}
 \left(\int_M f^{\frac{n}{n-2}}\right)^{\frac{n-2}{p_0}\cdot \frac{p_0}{n}} \notag
\\ 
&=\frac{2}{n} \varepsilon^{-(\frac{n-2}{n})^2} \int_M f+\frac{n-2}{n} \varepsilon^{\frac{2(n-2)}{n^2}} 
 \left(\int_M f^{\frac{n}{n-2}}\right)^{\frac{n-2}{n}}.  \notag
\end{align}
We also have 
\begin{equation} \label {5.2}
\int_M u^{p+1} \leq \left(\int_M u^{p_0/2}\right)^{2/p_0}
 \left(\int_M u^{p \cdot \frac{p_0}{p_0-2}}\right)^{\frac{p_0-2}{p_0}}
\end{equation}
for a nonnegative measurable function $u$ on $M$.

Now we consider a smooth solution of the Ricci flow satisfying the conditions of Lemma \ref{lem_4.1} (or Lemma \ref{lem_3.4}). Employing \eqref{2.3} and the above two inequalities with $u=|Rm|$ and $f=|Rm|^p$ we 
deduce for $p \geq 1$ and $t\in [0, T)$
\begin{align} \label{5.3}
\frac{\partial}{\partial t} \int_M |Rm|_{}^p &\leq -\frac{4(p-1)}{p} \int_M |\nabla |Rm|_{}^{p/2}|^2+c(n)p
\|Rm\|_{p_0/2}  \\
& \cdot \left(\varepsilon^{-(\frac{n-2}{n})^2} \int_M |Rm|_{}^p+\varepsilon^{\frac{2(n-2)}{n^2}} 
 \left(\int_M |Rm|_{}^{p \cdot\frac{n}{n-2}}\right)^{\frac{n-2}{n}}\right).\notag
\end{align}

Combining \eqref{5.3} with \eqref{3.1}, \eqref{3.19} and Lemma \ref{lem_3.4} we infer for $t \in [0, T_0] \cap [0, T)$, using the notations of \eqref{3.5}, that
\begin{align} \label{5.4}
&\frac{\partial}{\partial t} \int_M |Rm|_{}^p
\\
&\leq \left(-\frac{4(p-1)}{p}+ c(n,\gamma)p\varepsilon^{\frac{2(n-2)}{n^2}} \|Rm\|_{p_0/2} C_S^2(0)\right)\int_M |\nabla |Rm|_{}^{p/2}|^2\notag
\\
&\hspace{3mm}+c(n,\gamma)p\| Rm\|_{p_0/2} (\varepsilon^{-(\frac{n-2}{n})^2}+\varepsilon^{\frac{2(n-2)}{n^2}} ) \int_M |Rm|_{}^p  \notag \\
&\leq \left(-\frac{4(p-1)}{p}+ c(n,\gamma)p\varepsilon^{\frac{2(n-2)}{n^2}} \theta_0
(\frac{C_S^2(0)}{t}+1)^{2/n} \right)\int_M |\nabla |Rm|_{}^{p/2}|^2 \notag \\
& \hspace{3mm} +c(n,\gamma)p\theta_0 C_S^{-2}(0) (\frac{C_S^2(0)}{t}+1)^{2/n}(\varepsilon^{-(\frac{n-2}{n})^2}+\varepsilon^{\frac{2(n-2)}{n^2}} ) \int_M |Rm|_{}^p. \notag
\end{align}
We choose $\varepsilon$ to solve the equation
\[
c(n,\gamma)p\theta_0 (\frac{C_S^2(0)}{t}+1)^{2/n}\varepsilon^{\frac{2(n-2)}{n^2}}=\frac{p-1}{p}.
\]
Then we deduce
\begin{align} \label{5.5}
\frac{\partial}{\partial t} \int_M |Rm|_{}^p &\leq -\frac{3(p-1)}{p}\int_M |\nabla |Rm|_{}^{p/2}|^2 \\
&\hspace{3mm}+C_S^{-2}(0)c(n,\gamma)\left( p^{\frac{n}{2}} \theta_0^{\frac{n}{2}}(\frac{C_S^2(0)}{t}+1)+1 \right) \int_M |Rm|_{}^p\notag
\end{align}
with a new constant $c(n, \gamma)>0$.
For convenience, we abbreviate $C_S(0)$ to $C_S$ and define
\[
\Gamma_p(t)=c(n,\gamma)\left(p^{\frac{n}{2}} \theta_0^{\frac{n}{2}}(\frac{C_S^2}{t}+1)+1\right),
\]
where  $c(n, \gamma)$ is from \eqref{5.5}.

Next we consider $T' \in [0, T_0] \cap [0, T)$. Set for $0<\tau<\tau'\leq T'$  
\begin{equation*}
\psi(t)=\begin{cases}
0, & 0\leq t \leq \tau \\
\frac{t-\tau}{\tau'-\tau}, & \tau\leq t\leq \tau'\\
1, & \tau' \leq t \leq T'
\end{cases}
\end{equation*}

Then we have 
\begin{align} \label{5.6}
\frac{\partial}{\partial t} \left(\psi \int_M |Rm|_{}^p\right)+\psi \int_M |\nabla |Rm|_{}^{p/2}|^2 
\leq (\psi'+\psi C_S^{-2} \Gamma_p ) \int_M |Rm|_{}^p,
\end{align}
provided that $p \geq \frac{3}{2}$. Integration then yields for $\tau'< t \leq T'$
\begin{align} \label{5.7}
\int_{M} |Rm|_{}^p|_t + \int_{\tau'}^t \int_M |\nabla |Rm|_{}^{p/2}|^2 
&\leq \int_{\tau}^{T'} (\frac{1}{\tau'-\tau}+ C_S^{-2}\Gamma_p) \int_M |Rm|_{}^p.
\end{align}
Next we derive from \eqref{5.7}
\begin{align}
&\int_{\tau'}^{T'} \int_{M} |Rm|_{}^{p(1+\frac{2}{n})} 
\\
&\leq \int_{\tau'}^{T'} \left(\int_{M} |Rm|_{}^p \right)^{2/n}
\left( \int_{M} |Rm|_{}^{p \cdot \frac{n}{n-2}} \right)^{\frac{n-2}{n}} \notag
\\
&\leq c(n, \gamma) \sup_{\tau' \leq t\leq T'} \left(\int_{M} |Rm|_{}^p \right)^{2/n} \int_{\tau'}^{T'} (C_S^2 \int_M |\nabla |Rm|_{}^{p/2}|^2 
+\int_M |Rm|_{}^p) \notag \\
&\leq c(n,\gamma)(C_S^2+1) \left( \int_{\tau}^{T'} (\frac{1}{\tau'-\tau}+C_S^{-2}\Gamma_p)
\int_M |Rm|_{}^p \right)^{1+\frac{2}{n}}. \notag
\end{align}

Now we define
\[
H(p, \tau)=\int_{\tau}^{T'} \int_M |Rm|_{}^p
\]
and deduce 
\begin{equation} \label{5.8}
H(p(1+\frac{2}{n}), \tau') \leq  c(n, \gamma) C_S^{-4/n}  \left(\frac{C_S^2}{\tau'-\tau}+\Gamma_p(\tau)\right)^{1+\frac{2}{n}}H(p, \tau)^{1+\frac{2}{n}}.
\end{equation}
Set $\mu=1+\frac{2}{n}, q_k=q_0 \mu^k$ (recall $q_0=p_0/2$) and $\tau_k=(1-\frac{1}{\mu^{k+1}}) T'.$ Then we have
\[
\frac{1}{\tau_{k+1}-\tau_k}=\frac{\mu^{k+2}}{\mu-1} \cdot \frac{1}{T'}
\]
and 
\begin{align} 
\frac{C_S^2}{\tau_{k+1}-\tau_k}+\Gamma_{q_k}(\tau_k) &\leq c(n, \gamma) \mu^k q_k^{n/2} (\frac{1}{2}+\theta_0^{n/2})\frac{C_S^2}{T'} \notag \\
&\leq c(n, \gamma) \mu^k q_k^{n/2} \frac{C_S^2}{T'} \notag
\end{align}
with a new $c(n, \gamma)>0$. 
It follows that
\[
H(q_k(1+\frac{2}{n}), \tau_{k+1}) \leq  c(n, \gamma) C_S^{-4/n} \mu^{k(1+\frac{n}{2})}q_k^{\frac{n}{2}(1+\frac{n}{2})} 
\left(\frac{C_S^2}{T'}\right)^{1+\frac{2}{n}} H(q_k, \tau_k)^{1+\frac{2}{n}}.
\]
and hence
\[
H(q_{k+1}, \tau_{k+1})^{1/q_{k+1}} \leq c(n, \gamma)^{1/q_{k+1}} C_S^{-\frac{4}{n q_{k+1}} }\mu^{\frac{k}{q_{k}}} 
q_k^{\frac{n}{2q_k}}\left(\frac{C_S^2}{T'}\right)^{1/q_k} H(q_k, \tau_k)^{1/q_k}.
\]
There holds
\[
\sum_{k\geq 0} \frac{1}{q_{k+1}}=\frac{n-2}{n}, \, \,  \sum_{k\geq 0} \frac{1}{q_{k}}=1-\frac{4}{n^2}.
\]
Replacing $T'$ by $t>0$, iterating the above estimate and taking the  limit, we then arrive at 
\begin{equation} \label{5.9}
\sup_{M \times [(1-\frac{1}{\mu})t, t]}|Rm| \leq c(n, \gamma) C_S^{-\frac{4(n-2)}{n^2}}
 \left(\frac{C_S^2}{t}\right)^{1-\frac{4}{n^2}} \left( \int_{(1-\frac{1}{\mu})t}^t \int_M |Rm|_{}^{p_0/2}\right)^{2/p_0}
\end{equation}
By Lemma \ref{lem_3.4} we then infer
\begin{equation} \label{5.10}
\sup_{M \times [\frac{2}{n+2}t, t])}|Rm| \leq c(n, \gamma) 
 \frac{C_S^2}{t} \| Rm \|_{n/2}(0).
\end{equation}

Hence we have proved Lemma \ref{lem_4.1}.

\bibliographystyle{alpha}
\bibliography{references}

\begin{thebibliography}{{Gal}88b}

\bibitem[{Che}22]{Chen}
Eric {Chen}.
\newblock {Convergence of the Ricci flow on asymptotically flat manifolds with
  integral curvature pinching}.
\newblock {\em Ann. Sc. Norm. Super. Pisa Cl. Sci.}, 2022.
\newblock To appear.

\bibitem[CLN06]{Chow}
Bennett Chow, Peng Lu, and Lei Ni.
\newblock {\em Hamilton's {R}icci flow}, volume~77 of {\em Graduate Studies in
  Mathematics}.
\newblock American Mathematical Society, Providence, RI; Science Press Beijing,
  New York, 2006.

\bibitem[CWY21]{CWY2021}
Eric Chen, Guofang Wei, and Rugang Ye.
\newblock Ricci flow and a sphere theorem for {$L^{n/2}$}-pinched {Y}amabe
  metrics.
\newblock {\em Adv. Math.}, 393:Paper No. 108054, 14pp, 2021.

\bibitem[CWY22]{CWY21c}
Eric Chen, Guofang Wei, and Rugang Ye.
\newblock {Ricci flow, sphere theorems and other $L^{n/2}$-pinching theorems}.
\newblock {\em Preprint}, 2022.

\bibitem[DPW00]{DPW2000}
Xianzhe {Dai}, Peter {Petersen}, and Guofang {Wei}.
\newblock {Integral pinching theorems}.
\newblock {\em {Manuscr. Math.}}, 101(2):143--152, 2000.

\bibitem[{Gal}88a]{Gallot1988}
Sylvestre {Gallot}.
\newblock {In\'egalit\'es isop\'erimetriques et analytiques sur les vari\'etes
  riemanniennes. (Isoperimetric and analytic inequalities on Riemannian
  manifolds)}.
\newblock In {\em On the geometry of differentiable manifolds. Workshop, Rome,
  June 23-27, 1986}. Paris: Soci\'et\'e Math\'ematique de France, 1988.

\bibitem[{Gal}88b]{Gallot1988b}
Sylvestre {Gallot}.
\newblock {Isoperimetric inequalities based on integral norms of Ricci
  curvature}.
\newblock In {\em Colloque Paul L\'evy sur les processus stochastiques. (22-26
  juin 1987, Ecole Polytechnique, Palaiseau)}. Paris: Soci\'et'e Math\'ematique
  de France, 1988.

\bibitem[Gro78]{Gromov1978a}
M.~Gromov.
\newblock Almost flat manifolds.
\newblock {\em J. Differential Geometry}, 13(2):231--241, 1978.

\bibitem[Ham82]{Hamilton82}
Richard~S. Hamilton.
\newblock Three-manifolds with positive {R}icci curvature.
\newblock {\em J. Differential Geometry}, 17(2):255--306, 1982.

\bibitem[Ham86]{Hamilton86}
Richard~S. Hamilton.
\newblock Four-manifolds with positive curvature operator.
\newblock {\em J. Differential Geom.}, 24(2):153--179, 1986.

\bibitem[{Mos}66]{Moser}
J\"urgen {Moser}.
\newblock {A rapidly convergent iteration method and non-linear partial
  differential equations. I, II}.
\newblock {\em {Ann. Sc. Norm. Super. Pisa, Sci. Fis. Mat., III. Ser.}},
  20:265--315, 499--535, 1966.

\bibitem[{Per}02]{P}
Grisha {Perelman}.
\newblock {The entropy formula for the Ricci flow and its geometric
  applications}.
\newblock {\em arXiv Mathematics e-prints}, page math/0211159, November 2002.

\bibitem[PS98]{Petersen-Sprouse1998}
Peter Petersen and Chadwick Sprouse.
\newblock Integral curvature bounds, distance estimates and applications.
\newblock {\em J. Differential Geom.}, 50(2):269--298, 1998.

\bibitem[Ruh82]{Ruh1982}
Ernst~A. Ruh.
\newblock Almost flat manifolds.
\newblock {\em J. Differential Geometry}, 17(1):1--14, 1982.

\bibitem[Str16]{Streets16}
Jeffrey Streets.
\newblock A concentration-collapse decomposition for {$L^2$} flow
  singularities.
\newblock {\em Comm. Pure Appl. Math.}, 69(2):257--322, 2016.

\bibitem[Yan92]{Yang}
Deane Yang.
\newblock {$L^p$} pinching and compactness theorems for compact {R}iemannian
  manifolds.
\newblock {\em Forum Math.}, 4(3):323--333, 1992.

\bibitem[Ye15]{Ye15}
Rugang Ye.
\newblock The logarithmic {S}obolev and {S}obolev inequalities along the
  {R}icci flow.
\newblock {\em Commun. Math. Stat.}, 3(1):1--36, 2015.

\bibitem[Ye21]{Ye21}
Rugang Ye.
\newblock {Sobolev inequalities and diameter estimates for the Ricci flow}.
\newblock {\em Preprint}, 2021.

\end{thebibliography}

\end{document}